\documentclass[12pt]{amsart}
\usepackage{amsmath,amsthm,amsfonts,amssymb,latexsym}

\usepackage{hyperref}

\usepackage{enumerate}
\usepackage{color}
\usepackage[dvipsnames]{xcolor}

\usepackage[shortlabels]{enumitem}

\usepackage{parskip}
\begin{document}

\theoremstyle{plain}

\newtheorem{thm}{Theorem}[section]

\newtheorem{lem}[thm]{Lemma}
\newtheorem{Problem B}[thm]{Problem B}

\newtheorem{thml}{Theorem}
\renewcommand*{\thethml}{\Alph{thml}}

\newtheorem{pro}[thm]{Proposition}
\newtheorem{conj}[thm]{Conjecture}
\newtheorem{cor}[thm]{Corollary}
\newtheorem{que}[thm]{Question}
\newtheorem{rem}[thm]{Remark}
\newtheorem{defi}[thm]{Definition}
\newtheorem{hyp}[thm]{Hypotheses}

\newtheorem*{thmA}{Theorem A}
\newtheorem*{thmB}{Theorem B}
\newtheorem*{corB}{Corollary B}
\newtheorem*{thmC}{Theorem C}
\newtheorem*{thmD}{Theorem D}
\newtheorem*{thmE}{Theorem E}
 
\newtheorem*{thmAcl}{Theorem A$^{*}$}
\newtheorem*{thmBcl}{Theorem B$^{*}$}
\newcommand{\dd}{\mathrm{d}}

\newcommand{\Maxn}{\operatorname{Max_{\textbf{N}}}}
\newcommand{\Syl}{\operatorname{Syl}}
\newcommand{\Lin}{\operatorname{Lin}}
\newcommand{\U}{\mathbf{U}}
\newcommand{\R}{\mathbf{R}}
\newcommand{\dl}{\operatorname{dl}}
\newcommand{\Con}{\operatorname{Con}}
\newcommand{\cl}{\operatorname{cl}}
\newcommand{\Stab}{\operatorname{Stab}}
\newcommand{\Aut}{\operatorname{Aut}}
\newcommand{\Ker}{\operatorname{Ker}}
\newcommand{\InnDiag}{\operatorname{InnDiag}}
\newcommand{\fl}{\operatorname{fl}}
\newcommand{\Irr}{\operatorname{Irr}}
\newcommand{\FF}{\mathbb{F}}
\newcommand{\EE}{\mathbb{E}}
\newcommand{\normal}{\trianglelefteq}
\newcommand{\sn}{\normal\normal}
\newcommand{\Bl}{\mathrm{Bl}}
\newcommand{\NN}{\mathbb{N}}
\newcommand{\N}{\mathbf{N}}
\newcommand{\bfC}{\mathbf{C}}
\newcommand{\bfO}{\mathbf{O}}
\newcommand{\bfF}{\mathbf{F}}
\def\GGG{{\mathcal G}}
\def\HHH{{\mathcal H}}
\def\HH{{\mathcal H}}
\def\irra#1#2{{\rm Irr}_{#1}(#2)}

\renewcommand{\labelenumi}{\upshape (\roman{enumi})}

\newcommand{\PSL}{\operatorname{PSL}}
\newcommand{\PSU}{\operatorname{PSU}}
\newcommand{\alt}{\operatorname{Alt}}

\providecommand{\V}{\mathrm{V}}
\providecommand{\E}{\mathrm{E}}
\providecommand{\ir}{\mathrm{Irm_{rv}}}
\providecommand{\Irrr}{\mathrm{Irm_{rv}}}
\providecommand{\re}{\mathrm{Re}}

\numberwithin{equation}{section}
\def\irrp#1{{\rm Irr}_{p'}(#1)}

\def\ibrrp#1{{\rm IBr}_{\Bbb R, p'}(#1)}
\def\C{{\mathbb C}}
\def\Q{{\mathbb Q}}
\def\irr#1{{\rm Irr}(#1)}
\def\irrp#1{{\rm Irr}_{p^\prime}(#1)}
\def\irrq#1{{\rm Irr}_{q^\prime}(#1)}
\def \c#1{{\cal #1}}
\def \aut#1{{\rm Aut}(#1)}
\def\cent#1#2{{\bf C}_{#1}(#2)}
\def\norm#1#2{{\bf N}_{#1}(#2)}
\def\zent#1{{\bf Z}(#1)}
\def\syl#1#2{{\rm Syl}_#1(#2)}
\def\normal{\triangleleft\,}
\def\oh#1#2{{\bf O}_{#1}(#2)}
\def\Oh#1#2{{\bf O}^{#1}(#2)}
\def\det#1{{\rm det}(#1)}
\def\gal#1{{\rm Gal}(#1)}
\def\fit#1{{\bf F}(#1)}
\def\ker#1{{\rm ker}(#1)}
\def\normalm#1#2{{\bf N}_{#1}(#2)}
\def\alt#1{{\rm Alt}(#1)}
\def\iitem#1{\goodbreak\par\noindent{\bf #1}}
   \def \mod#1{\, {\rm mod} \, #1 \, }
\def\sbs{\subseteq}

\def\gc{{\bf GC}}
\def\ngc{{non-{\bf GC}}}
\def\ngcs{{non-{\bf GC}$^*$}}
\newcommand{\notd}{{\!\not{|}}}

\newcommand{\Z}{\mathbf{Z}}
\newcommand{\Out}{{\mathrm {Out}}}
\newcommand{\Mult}{{\mathrm {Mult}}}
\newcommand{\Inn}{{\mathrm {Inn}}}
\newcommand{\IBR}{{\mathrm {IBr}}}
\newcommand{\IBRL}{{\mathrm {IBr}}_{\ell}}
\newcommand{\IBRP}{{\mathrm {IBr}}_{p}}
\newcommand{\cd}{\mathrm{cd}}
\newcommand{\ord}{{\mathrm {ord}}}
\def\id{\mathop{\mathrm{ id}}\nolimits}
\renewcommand{\Im}{{\mathrm {Im}}}
\newcommand{\Ind}{{\mathrm {Ind}}}
\newcommand{\diag}{{\mathrm {diag}}}
\newcommand{\soc}{{\mathrm {soc}}}
\newcommand{\End}{{\mathrm {End}}}
\newcommand{\sol}{{\mathrm {sol}}}
\newcommand{\Hom}{{\mathrm {Hom}}}
\newcommand{\Mor}{{\mathrm {Mor}}}
\newcommand{\Mat}{{\mathrm {Mat}}}
\def\rank{\mathop{\mathrm{ rank}}\nolimits}
\newcommand{\Tr}{{\mathrm {Tr}}}
\newcommand{\tr}{{\mathrm {tr}}}
\newcommand{\Gal}{{\rm Gal}}
\newcommand{\Spec}{{\mathrm {Spec}}}
\newcommand{\ad}{{\mathrm {ad}}}
\newcommand{\Sym}{{\mathrm {Sym}}}
\newcommand{\Char}{{\mathrm {Char}}}
\newcommand{\pr}{{\mathrm {pr}}}
\newcommand{\rad}{{\mathrm {rad}}}
\newcommand{\abel}{{\mathrm {abel}}}
\newcommand{\PGL}{{\mathrm {PGL}}}
\newcommand{\PCSp}{{\mathrm {PCSp}}}
\newcommand{\PGU}{{\mathrm {PGU}}}
\newcommand{\codim}{{\mathrm {codim}}}
\newcommand{\ind}{{\mathrm {ind}}}
\newcommand{\Res}{{\mathrm {Res}}}
\newcommand{\Lie}{{\mathrm {Lie}}}
\newcommand{\Ext}{{\mathrm {Ext}}}
\newcommand{\Alt}{{\mathrm {Alt}}}
\newcommand{\AAA}{{\sf A}}
\newcommand{\SSS}{{\sf S}}
\newcommand{\DDD}{{\sf D}}
\newcommand{\QQQ}{{\sf Q}}
\newcommand{\CCC}{{\sf C}}
\newcommand{\SL}{{\mathrm {SL}}}
\newcommand{\Sp}{{\mathrm {Sp}}}
\newcommand{\PSp}{{\mathrm {PSp}}}
\newcommand{\SU}{{\mathrm {SU}}}
\newcommand{\GL}{{\mathrm {GL}}}
\newcommand{\GU}{{\mathrm {GU}}}
\newcommand{\Spin}{{\mathrm {Spin}}}
\newcommand{\CC}{{\mathbb C}}
\newcommand{\CB}{{\mathbf C}}
\newcommand{\RR}{{\mathbb R}}
\newcommand{\QQ}{{\mathbb Q}}
\newcommand{\ZZ}{{\mathbb Z}}
\newcommand{\bfN}{{\mathbf N}}
\newcommand{\bfZ}{{\mathbf Z}}
\newcommand{\PP}{{\mathbb P}}
\newcommand{\cG}{{\mathcal G}}
\newcommand{\cH}{{\mathcal H}}
\newcommand{\cQ}{{\mathcal Q}}
\newcommand{\GA}{{\mathfrak G}}
\newcommand{\cT}{{\mathcal T}}
\newcommand{\cL}{{\mathcal L}}
\newcommand{\IBr}{\mathrm{IBr}}
\newcommand{\cS}{{\mathcal S}}
\newcommand{\cR}{{\mathcal R}}
\newcommand{\GCD}{\GC^{*}}
\newcommand{\TCD}{\TC^{*}}
\newcommand{\FD}{F^{*}}
\newcommand{\GD}{G^{*}}
\newcommand{\HD}{H^{*}}
\newcommand{\GCF}{\GC^{F}}
\newcommand{\TCF}{\TC^{F}}
\newcommand{\PCF}{\PC^{F}}
\newcommand{\GCDF}{(\GC^{*})^{F^{*}}}
\newcommand{\RGTT}{R^{\GC}_{\TC}(\theta)}
\newcommand{\RGTA}{R^{\GC}_{\TC}(1)}
\newcommand{\Om}{\Omega}
\newcommand{\eps}{\epsilon}
\newcommand{\varep}{\varepsilon}
\newcommand{\al}{\alpha}
\newcommand{\chis}{\chi_{s}}
\newcommand{\sigmad}{\sigma^{*}}
\newcommand{\PA}{\boldsymbol{\alpha}}
\newcommand{\gam}{\gamma}
\newcommand{\lam}{\lambda}
\newcommand{\la}{\langle}
\newcommand{\genf}{F^*}
\newcommand{\ra}{\rangle}
\newcommand{\hs}{\hat{s}}
\newcommand{\htt}{\hat{t}}
\newcommand{\tG}{\hat G}
\newcommand{\St}{\mathsf {St}}
\newcommand{\bfs}{\boldsymbol{s}}
\newcommand{\bfl}{\boldsymbol{\lambda}}
\newcommand{\tn}{\hspace{0.5mm}^{t}\hspace*{-0.2mm}}
\newcommand{\ta}{\hspace{0.5mm}^{2}\hspace*{-0.2mm}}
\newcommand{\tb}{\hspace{0.5mm}^{3}\hspace*{-0.2mm}}
\def\skipa{\vspace{-1.5mm} & \vspace{-1.5mm} & \vspace{-1.5mm}\\}
\newcommand{\tw}[1]{{}^#1\!}
\renewcommand{\mod}{\bmod \,}

\marginparsep-0.5cm

\renewcommand{\thefootnote}{\fnsymbol{footnote}}
\footnotesep6.5pt
\title{The blocks with five irreducible characters}

\author[]{J. Miquel Mart\'inez}
\address{J. Miquel Mart\'inez, Departament de Matem\`atiques, Universitat de Val\`encia}
\email{josep.m.martinez@uv.es}

\author[]{Noelia Rizo}
\address{Noelia Rizo, Departament de Matem\`atiques, Universitat de Val\`encia}
\email{noelia.rizo@uv.es}

\author[]{Lucia Sanus}
\address{Lucia  Sanus, Departament de Matem\`atiques, Universitat de Val\`encia}
\email{lucia.sanus@uv.es}

\thanks{This research is partially supported by the
Spanish Ministerio de Ciencia
e Innovaci\'on
(Grant PID2019-103854GB-I00 funded by MCIN/AEI/10.13039/501100011033) and by Generalitat Valenciana CIAICO/2021/163. The first author is supported by a fellowship UV-INV-PREDOC20-1356056 from Universitat de Val\`encia. The second author is supported by a CDEIGENT grant CIDEIG/2022/29 funded by Generalitat Valenciana.}



\begin{abstract} 
Let $G$ be a finite group, $p$ a prime and $B$ a Brauer $p$-block of $G$ with defect group $D$. We prove that if the number of irreducible ordinary characters in $B$ is $5$ then $D\cong \CCC_5, \CCC_7, \DDD_8$ or $\QQQ_8$, assuming that the Alperin--McKay conjecture holds for $B$.
 \end{abstract}

\maketitle

\section{Introduction}

Suppose that $G$ is a finite group, $p$ is a prime, and $B$ is a Brauer $p$-block of $G$ with defect group $D$. Let $k(B)$ denote the number of irreducible ordinary characters that belong to $B$. In \cite{BF59}, R. Brauer and W. Feit proved that 
$$k(B)\leq \frac{1}{4}|D|^{2}+1.$$ Besides this bound, very little is known about the number of irreducible complex characters in the block $B$. However, much has been conjectured. For instance, Brauer's $k(B)$-conjecture claims that $k(B)\leq |D|$, and the H\'ethelyi--K\"ulshammer conjecture posits that if $k(B)>1$ then $k(B)\geq 2\sqrt{p-1}$ (see \cite{brauer} and \cite{HK}). Remarkably, neither of these conjectures has even been reduced to simple groups.

Another fundamental problem on the number $k(B)$ is Brauer's problem 21, which asks whether for any positive integer $n$ there are finitely many isomorphism classes of prime power order groups that occur as defect groups of blocks with exactly $n$ irreducible characters. This problem was shown to have a positive answer for $p$-solvable groups in \cite{Kul90}. Using this result and E. Zelmanov's solution to the restricted Burnside problem, it was shown to have a positive answer as a consequence of the Alperin--McKay conjecture in \cite{KR}. However, no classification of the possible defect groups is given for any value of $k(B)$.

R. Brauer and C. Nesbitt proved in 1941 that if $k(B)=1$ then $|D|=1$ (see \cite{BN41}). The next step came 41 years later, when J. Brandt proved in \cite{Bra} that $k(B)=2$ if and only if $|D|=2$. If $B$ is the principal block, the classifications of defect groups for $k(B)\leq 6$ have been obtained in \cite{koshitani}, \cite{rsv} and \cite{HSV23} but at the present time there is no classification for arbitrary blocks.

It is however a consequence of the Alperin--McKay conjecture that $k(B)=3$ if and only if $|D|=3$, as shown in \cite[Theorem 4.2]{KNST}. In \cite{MRS23} we proved that if $k(B)=4$ then $D\cong\CCC_2\times\CCC_2, \CCC_4,$ or $\CCC_5$ also assuming that the Alperin--McKay conjecture holds for $B$. The main result of this paper is to extend these results.

\begin{thml}\label{thm:A}
Suppose that $B$ is a Brauer $p$-block of a finite group $G$
with defect group $D$. Assume that $k(B)=5$. If the
Alperin--McKay conjecture holds for $B$, then $D$ is isomorphic to $\CCC_5, \CCC_7,\DDD_8$ or $\QQQ_8$.
\end{thml}

The proof of Theorem \ref{thm:A} is achieved via a reduction to a problem on primitive permutation groups of low rank with a small number of projective characters, in the sense of Schur. To solve this problem we use known classifications of these groups, and use some very \textit{ad hoc} arguments for each case in the classifications. Further, we care to remark that the defect groups that we find coincide with those found for principal blocks in \cite{rsv}. %Examples of nonprincipal blocks with $5$ irreducible characters can be found in the dihedral groups $\DDD_{28}$ and $\DDD_{30}$ for $p=7$ and $p=5$ respectively, as well as in  $\mathrm{SmallGroup}(48,15)$ and  $\mathrm{SmallGroup}(48,16)$ for $p=2$. 

There is a complementary conjecture to Brauer's problem 21, known as Donovan's conjecture. It asks whether, for a finite $p$-group $D$, there are finitely many Morita equivalence classes of blocks with defect group isomorphic to $D$. Although this conjecture remains open, there has been a recent breakthrough for abelian defect groups \cite{EL} when $p=2$. 
Assuming the Alperin--McKay conjecture and using the progress in Donovan's conjecture that is gathered in the Block library {{\cite{wiki}}}, it is possible to determine the Morita equivalence classes of blocks with $k(B)\in\{4, 5\}$ by using \cite[Theorem A]{MRS23} and Theorem \ref{thm:A}.

%Finally, we mention that using our results and the recent advances on Donovan's conjecture, it is possible to predict exactly the Morita equivalence classes of blocks with exactly 4 or 5 irreducible characters (see Conjectures \ref{conj:kb4} and \ref{conj:kb5}). \textcolor{magenta}{I think Miquel had a beautiful phrasing for this in mind $\heartsuit$}

Section \ref{sec:red} is devoted to the proof of the reduction theorem. In Section \ref{sec:lemmas} we show some preliminary general results on projective characters for the proof of the reduced statement and in Section \ref{sec:proof} we provide the results and known classifications of primitive permutation groups that we need and complete the proof of the main result. 

%Finally, the discussion of the possible Morita equivalence classes of blocks with 4 or 5 irreducible characters is done in Section \ref{sec:Morita}

\noindent {\bf Acknowledgements.}
%The authors would like to thank Gabriel Navarro for many useful conversations on this topic. 
Part of this work was done while the first named author visited the Department of Mathematics at the  Rheinland-Pf\"alzische Technische Universit\"at (formerly TU Kaiserslautern) funded by a travel grant associated to UV-INV-PREDOC20-1356056. He thanks the entire department for their warm hospitality and specially Gunter Malle for supervising his visit and for a thorough read of an earlier version of this manuscript and suggestions that helped improve the exposition.  This research was also carried out during a visit by the third named author to the Dipartimento di Matematica e Informatica (DIMAI) of Universit\` a degli Studi di Firenze. She would like to thank DIMAI for their hospitality. The authors also thank Benjamin Sambale for a thorough read of an earlier version of this manuscript and suggestions that helped improve the exposition of this work.

\section{The reduction}\label{sec:red}

In this paper we follow the notation of \cite{I08} for finite groups, \cite{N18} for ordinary characters and \cite{N98} for Brauer characters and blocks. 

\subsection{Preliminary results} 

We begin by stating some necessary results for the reduction theorem. First, we recall the known results on small values of $k(B)$. We denote by ${\rm IBr}(B)$ the set of irreducible Brauer characters in the $p$-block $B$ and by $l(B)=|\IBr(B)|$. We also denote by $k_0(B)$ the number of height zero irreducible characters in $B$.
\begin{thm}\label{thm:kb small}
Let $G$ be a finite group, let $B$ a $p$-block of $G$ with defect group $D$. Then
\begin{enumerate}
\item  $k(B)=1$ if and only if $|D|=1$,
\item  $k(B)=2$ if and only if $|D|=2$,
\item if the Alperin--McKay conjecture holds for $B$ and $k(B)=3$ then $|D|=3$,
\item if the Alperin--McKay conjecture holds for $B$ and $k(B)=4$ then $|D|\in \{4,5\}$,
\item if $l(B)=1$ and $k(B)=5$ then $D\cong \CCC_5, \QQQ_8, \DDD_8$.
\end{enumerate}
\end{thm}
\begin{proof} 
Part (i) is \cite[Theorem 3.18]{N98}, (ii) is \cite{Bra}, (iii) is \cite[Theorem 4.2]{KNST}, (iv) is \cite{MRS23} and (v) appears in \cite{CK92}.
\end{proof}

%\begin{lem}\label{lem:special 2-group}
%Suppose that $D$ is a $2$-group such that $D'=\zent D=\Phi(D)$ and assume $D/D'\cong\CCC_2\times \CCC_2$. Then $D\cong \QQQ_8, \DDD_8$.
%\end{lem}

\begin{lem}\label{lem:special 2-group}
Suppose that $D$ is a $2$-group such that $D'=\zent D$ and assume $D/D'\cong\CCC_2\times \CCC_2$. Then $D\cong \QQQ_8$ or $\DDD_8$.
\end{lem}
\begin{proof}
Since $D/D'\cong \CCC_2 \times \CCC_2$ then $D$ is dihedral, semidihedral or generalized quaternion. Thus $|\zent D|=2$ and $|D|=8$, as desired.
\end{proof}

 % By Burnside's Basis Theorem (see \cite[III.3.15]{huppert}), $D$ contains a set of generators with $2$ elements. Since $\zent D=D'$ we have that the third term of the lower central series is trivial. Now, by \cite[III.1.11(c)]{huppert} have that $D'$ is cyclic.  Then  $|D|=8$ and $D$ is nonabelian, so $D$ is isomorphic to either  $\DDD_8$ or $ \QQQ_8$.

We end the preliminaries by stating Brauer's famous character counting formula. Throughout this work we write $n_q$ for the largest power of a prime $q$ that divides the natural number $n$.

\begin{lem}[Brauer's formula]\label{lem:brauer formula}
Let $G$ be a finite group, and $B$ a $p$-block of $G$. Let $\{x_1,\dots,x_t\}$ be representatives of the non-central $G$-conjugacy classes of $p$-elements. Then
$$k(B)=l(B)|\zent G|_p+\sum_{i=1}^{t}\sum_{b\in\Bl(\cent G {x_i})\atop b^G=B}l(b).$$
\end{lem}
\begin{proof}
See \cite[Theorem 5.12]{N98}.
\end{proof}

\subsection{The reduction theorem}

We now show that our main result can be proved assuming the following theorem on affine primitive permutation groups, whose proof can be found in Section \ref{sec:proof of thm reduced}.

\begin{thm}\label{thm:reduced}
Let $G$ be a finite group, $p\geq 5$ a prime, $D\in\Syl_p(G)$, and assume $D$ is a minimal normal subgroup of $G$ with $p$-complement $K$. Let $Z=\zent{G}$ and suppose $Z=\oh{p'}G$. Assume that $G/Z$ is an affine primitive permutation group of rank $2, 3$ or $4$ on $D$. Let $\lambda\in\Irr(Z)$. If $|\Irr(G|\lambda)|=5$ we have $|D|=5,7$.
\end{thm}

Next we prove that Theorem \ref{thm:reduced} implies Theorem \ref{thm:A} when $D\normal G$.

\begin{thm}\label{thm:reduction}
Let $G$ be a finite group, $B$ a $p$-block of $G$ with $k(B)=5$. Let $D$ be a defect group of $B$ and assume $D\normal G$. Then $D$ is isomorphic to ${\sf{C}}_5,{\sf{C}}_7, {\sf{D}}_8$ or $ {\sf{Q}}_8$.
\end{thm}

\begin{proof}  Write $|D|=p^d$.

\textit{Step 1. We may assume that $D\in{\rm Syl}_p(G)$.  Further, we may assume that $D$ is a minimal normal subgroup of $G$ with $p\neq2, 3$. }

The first part is \cite[Theorem 6]{reynolds}. Suppose that $N\normal G$ with  $1<N<D$. Let  $\overline{B}$ be a block of $G/N$ dominated by $B$ with defect group $D/N$ (see \cite[Theorem 9.9]{N98}). Then $1<k(\overline{B})< 5$ by \cite[Theorem 3.18]{N98} and \cite[Theorem 6.10]{N98}.  
By Theorem \ref{thm:kb small}, $D/N$ is isomorphic to $\CCC_2, \CCC_3,\CCC_2\times \CCC_2,   \CCC_4$ or $\CCC_5$, and hence $p\leq 5$.  

%Suppose that $D$ is non-abelian. Then by setting $N=\zent D$ we have that necessarily $D/\zent D\cong \CCC_2\times \CCC_2$. Since $D'\sbs\Phi(D)\sbs N$, arguing similarly, we have $N=D'=\Phi(D)$, so we are under the hypotheses of Lemma \ref{lem:special 2-group} and thus $D$ is isomorphic to either  $\DDD_8$ or $\QQQ_8$.
Suppose that $D$ is non-abelian. Then by setting $N=\zent D$ we have that necessarily $D/\zent D\cong \CCC_2\times \CCC_2$. Since $D'\sbs\Phi(D)\sbs N$, arguing similarly, we have $N=D'$, so we are under the hypotheses of Lemma \ref{lem:special 2-group} and thus $D$ is isomorphic to either  $\DDD_8$ or $\QQQ_8$.

Thus we may assume that $D$ is abelian. By \cite[Theorem 9]{reynolds}, $k_0(B)=k(B)=5$. By \cite[Corollary 1.3]{landrock} and \cite[Corollary 1.6]{landrock} we have that $p\not=2,3$. Then $D/N\cong {\sf C}_5$  and $k(\overline{ B})=4$ by Theorem \ref{thm:kb small}. Since  all the irreducible characters in ${\rm Irr}(D)\setminus{\rm Irr}(D/N)$ lie under the same irreducible character of $B$, they are $G$-conjugate. Let $\theta\in\Irr(D)\setminus\Irr(D/ N)$. Then $|G:G_\theta|=5^d-5$ which is divisible by $5$, a contradiction with the fact that $D$ is a Sylow subgroup.  Hence  we may assume that $D$ is  a minimal normal subgroup of $G$. In particular, $D$ is abelian and again by \cite[Corollary 1.3]{landrock} and \cite[Corollary 1.6]{landrock} we may assume $p\geq 5$.

\smallskip

\textit{Step 2. We may assume that $|\zent G|_p=1$.}

\smallskip

Let $Z\in\Syl_p(\zent G)$ and suppose that $1<Z$.  Then $D=Z$ by Step 1, and $G$ has a normal $p$-complement $K$, that is, $G=K\times Z$. In this case, by \cite[Theorem 10.20]{N98}, $\Irr(B)=\Irr(K\times Z|\lambda)$ for some $\lambda\in\Irr(K)$ (which is $G$-invariant because $Z$ is central). Therefore $|\Irr(B)|=|\Irr(K\times Z|\lambda)|=|\Irr(Z)|$ which forces $Z=D\cong \CCC_5$.

\smallskip

\textit{Step 3. We may assume that $\zent G=\oh{p'}G$.}

\smallskip

Let $N={\rm \textbf{O}}_{p'}(G)$ and let $\lambda\in{\rm Irr}(N)$ be such that if $b=\{\lambda\}\in{\rm Bl}(N)$ then $B$ covers $b$. By the Fong--Reynolds correspondence \cite[Theorem 9.14]{N98} we may assume that $b$ is $G$-invariant, and hence $\lambda$ is $G$-invariant. Now $(G,N,\lambda)$ is an ordinary-modular character triple and by \cite[Problem 8.13]{N98} there exists an isomorphic ordinary-modular character triple $(H,M,\varphi)$ with $M$ a $p'$-group and $\varphi$ linear and faithful, in particular $M$ is central. Since $G/N\cong H/M$ we have that $M=\oh {p'}H$ and by \cite[Theorem 10.20]{N98} applied twice we obtain that
$$k(B)=|{\rm Irr}(G|\lambda)|=|{\rm Irr}(H|\varphi)|=k(B_1),$$ where $B_1$ is some $p$-block of $H$. Moreover, if $Q/M\in{\rm Syl}_p(H/M)$, since $M\subseteq\zent H$, we have that $Q=Q_1\times M$ where $Q_1=\oh p H$, so $Q_1$ is a defect group of $B_1$ and
$$D\cong DN/N\cong Q/M\cong Q_1.$$

Hence may assume that $\lambda$ is linear and faithful, in particular $N$ is central and by Step 2 we have $N=\zent G$.

\smallskip
\textit{Step 4. Let $\{x_1,\ldots,x_t\}$ be a set of representatives of the  $G$-conjugacy classes of $p$-elements of $G$. Then $t\in\{2,3,4\}$. }

\smallskip

We may assume that $x_1=1$, so $\{x_2,\ldots,x_t\}$ are a set of representatives of the non-trivial $G$-conjugacy classes of $p$-elements of $G$. By Brauer's formula from Lemma \ref{lem:brauer formula} and Step 2 we have

\begin{equation}\label{formula}
k(B)=l(B)+\sum_{i=2}^t\sum_{b\in \Bl(\cent G {x_i}),\\ b^G=B}l(b).
\end{equation}

Notice that we may assume $l(B)>1$ by Theorem \ref{thm:kb small}. Thus $l(B)\geq 2$ and $b^G=B$ for some $b\in{\rm Bl}(\cent G {x_i})$ by \cite[Theorem 4.14]{N98}, and we have that either $t\in\{2,3,4\}$, concluding the proof of the step.

\bigskip

Before the next step of our proof we fix some notation. Let $Z=\zent G=\oh{p'}G$. By \cite[Theorem 10.20]{N98} we have that $\Irr(B)=\Irr(G|\lambda)$, where $\lambda\in\Irr(Z)$ is the character from Step 3. Let $\hat{\lambda}=1_D\times \lambda\in\Irr(D\times Z)$ and let $b\in{\rm Bl}(D\times Z)$ be the $p$-block containing $\hat{\lambda}$. Notice that $b$ is $G$-invariant since $\hat{\lambda}$ is $G$-invariant.

\bigskip

\textit{Step 5. We have that $G/Z$ is an affine primitive permutation group of rank 2, 3 or 4.}

\smallskip

By  the Schur--Zassenhaus theorem there is $K\leq G$ such that $G=KD$ and $K\cap D=1$, and write $\overline{G}=G/Z$, $\overline{D}=DZ/Z$ and $\overline{K}=K/Z$. We have that $\overline{G}$ acts on $\Omega=\{\overline{K}d \mid d\in D\}$ transitively via the action of right multiplication. Notice that $\overline{K}$ is the stabilizer of the trivial class in $\Omega$. Since $\overline{K}$ is maximal in $\overline{G}$, we have that this action is primitive (see \cite[Corollary 8.14]{I08}, for instance) and $\overline{G}$ is a primitive permutation group.

Since the action of $\overline{K}$ on $\Omega$ has the same number of orbits (and orbit sizes) as the action by conjugation of $\overline{G}$ on $\overline{D}$, by Step 4 it has 2, 3 or 4 orbits. Hence the rank of $\overline{G}$ is 2, 3 or 4, as wanted. 

\medskip

Now we apply Theorem \ref{thm:reduced} to get the desired result. \end{proof}

In the final step of the previous proof we have considered an action on the right cosets of some subgroup. However, from now on, every action considered will be an action by conjugation on some normal subgroup, or on its set of irreducible characters.

We are now ready to prove Theorem \ref{thm:A}, which we restate. Recall that if $B$ is a block of a finite group $G$ with defect group $D$ then the Alperin--McKay conjecture states that $k_0(B)=k_0(b)$ where $b$ is its Brauer correspondent block in $\norm G D$.

\begin{cor}
Let $G$ be a finite group, $B$ a $p$-block of $G$ with $k(B)=5$ and $D$ a defect group of $B$. Assume $k_0(B)=k_0(b)$ where $b\in\Bl(\norm G D)$ is the Brauer correspondent of $B$ in $\norm G D$. Then $D$ is isomorphic to $\CCC_5, \CCC_7, \DDD_8$ or $\QQQ_8$.
\end{cor}

\begin{proof}
By \cite[Theorem 9.9]{N98}, $b$ dominates some block $\overline{b}\in\Bl(\norm G D/\Phi(D))$ with defect group $D/\Phi(D)$. We have $k(\overline{b})=k_0(\overline{b})\leq k_0(b)$ by \cite[Theorem 9]{reynolds}. We explore all the possibilities for $k(\overline{b})$. If $k(\overline{b})=1$ then $D/\Phi(D)=1$ which is impossible. If $k(\overline{b})=2$ then by the main result of \cite{Bra}, $|D/\Phi(D)|=2$ and then $D$ is a cyclic 2-group. In particular, $k(B)=k_0(B)$ by \cite{dadecyclic}, which contradicts \cite[Corollary 1.3]{landrock}. If $k(\overline{b})=3$ by \cite[Theorem 4.1]{KNST} we have $|D/\Phi(D)|=3$ so $D$ is a cyclic 3-group. Again $k(B)=k_0(B)$ which contradicts \cite[Corollary 1.6]{landrock}. 

If $k(\overline{b})=4$ by the main result of \cite{MRS23} we have $D/\Phi(D)\in\{\mathsf{C}_4,\mathsf{C}_2\times\mathsf{C}_2,\mathsf{C}_5\}$. The first case is impossible. If $|D/\Phi(D)|=5$ then $D$ is cyclic, so applying \cite{dadecyclic}, $k(b)=k_0(b)=k_0(B)=k(B)=5$ and it follows from Theorem \ref{thm:reduced} applied to $b$ that $|D|=5$. Thus we assume $D/\Phi(D)\cong \mathsf{C}_2\times \mathsf{C}_2$, and notice that by \cite[Corollary 1.3]{landrock}, $D$ is nonabelian, and $k_0(B)=4$. By arguing similarly, we get that $D/D'\cong\mathsf{C}_2\times\mathsf{C}_2$ so $D$ is semidihedral, dihedral or generalized quaternion. Since these groups are metacyclic we may use \cite[Theorem 8.1]{sambale} to conclude that the only possibilities are $D\cong\mathsf{Q}_8$ and $D\cong\mathsf{D}_8$. 
%
%
%Since these groups are metacyclic we may use \cite[Theorem 8.1]{sambale} to conclude that the only possibilities are $D\cong\mathsf{Q}_8$ and $D\cong\mathsf{D}_8$.

%\cong D/\zent D\cong\mathsf{C}_2\times\mathsf{C}_2$, so by Lemma \ref{lem:special 2-group} we have that $D$ is isomorphic to either  $\DDD_8$ or $ \QQQ_8$, as wanted.

%$D$ is semidihedral, dihedral or generalized quaternion. 
%
%
%Since these groups are metacyclic we may use \cite[Theorem 8.1]{sambale} to conclude that the only possibilities are $D\cong\mathsf{Q}_8$ and $D\cong\mathsf{D}_8$. 

Finally if $k(\overline{b})=5$, we have that $5=k(\overline{b})\leq k_0(b)=k_0(B)\leq 5$, so $k(\overline{b})=k_0(b)=5$. By Theorem \ref{thm:reduction} applied to $\norm G D/\Phi(D)$ we have that $D/\Phi(D)$ is cyclic of order 5 or 7. Then $D$ is abelian and $k(b)=k_0(b)=5$ by \cite[Theorem 6]{reynolds}. By Theorem \ref{thm:reduction} applied to $\norm G D$, we obtain that $D$ is cyclic of order 5 or 7, concluding the proof.
\end{proof}

%\textcolor{blue}{It is possible that the second paragraph of the previous result can be argued using Lemma \ref{lem:special 2-group}} \textcolor{magenta}{I think it might be a problem if $D/\zent D\cong D_8,Q_8$. I leave the original argument for the moment}

\section{Some results on projective characters}\label{sec:lemmas}

We begin this section with the following observation.
\begin{lem}\label{lem:l(B)}
Let $G$ be a finite group and suppose that $G$ has a normal Sylow $p$-subgroup $D$ with $p$-complement $K$. Let $B\in\Bl(G)$ and suppose that there exists $\lambda\in{\rm Irr}(\oh{p'}G)$, $G$-invariant lying under some irreducible character in $B$. Then we have $l(B)=|\Irr(K|\lambda)|$.
\end{lem}
\begin{proof}
By \cite[Lemma 2.32]{N98} we have $\IBr(G)=\IBr(G/D)$ which can be identified with $\IBr(K)=\Irr(K)$ and the result follows.
\end{proof}
%By \cite[Theorem 10.20]{N98} we have that $\Irr(B)=\Irr(G|\lambda)$ and $\IBr(B)=\IBr(G|\lambda)$. Since $D=\oh{p}G$, by \cite[Lemma 2.32]{N98} then $\IBr(G)=\IBr(G/D)=\Irr(G/D)$ and $\IBr(B)=\IBr(G|\lambda\times 1_D)=\Irr(G|\lambda\times 1_D)$. Now by \cite[Thorem 1.18]{N18}, restriction defines a bijection $\Irr(G/D)\rightarrow\Irr(K)$ and this gives a bijection 
%$\Irr(G|\lambda\times 1_D)\rightarrow \Irr(K|\lambda)$ as desired.

In proving Theorem \ref{thm:reduced} we will be dealing with groups with few projective characters. The next results will help us deal with these, by allowing us to bring Sylow subgroups into the picture.

\begin{lem}\label{lem:cyclic}
Let $N\normal G$, $\theta\in\Irr(N)$ and assume $\theta$ is $G$-invariant and that $G/N$ is cyclic. Then $\theta$ extends to $G$.
\end{lem}
\begin{proof}
See \cite[Theorem 5.1]{N18}.
\end{proof}

We denote by ${\rm Lin}(G)$ the set of linear characters of $G$.

\begin{lem}\label{lem:schur}
Let $G$ be a finite group. Suppose that  $N\normal G$ and assume $\lambda\in\Irr(N)$ is $G$-invariant and linear. Let $o(\lambda)$ be the order of $\lambda$ as an element of $\Lin(N)$. If every Sylow $p$-subgroup of $G/N$ has trivial Schur multiplier whenever $p$ divides $o(\lambda)$ then $\lambda$ extends to $G$.
\end{lem}
\begin{proof}
This follows from \cite[Theorem 6.26]{I06} and \cite[Theorem 11.7]{I06}.
\end{proof}

If $N\normal G$, $\lambda\in\Irr(N)$ is $G$-invariant and $|\Irr(G|\lambda)| =1$ then we say $\lambda$ is fully ramified in $G$. In this case, it is easy to see that $|G:N|$ must be a square. One of the first applications of the classification of finite simple groups to character theory was proving that, in this case, $G/N$ is solvable (this is known as the Howlett--Isaacs theorem, see \cite{HI}).

\begin{lem}\label{lem:demeyerjanusz}
Let $Z\normal G$ and let $\lambda\in{\rm Irr}(Z)$ be $G$-invariant. Then $\lambda$ is fully ramified in $P$ for every $P/Z\in\Syl_p(G/Z)$ and every prime $p$ dividing $|G|$ if and only if $\lambda$ is fully ramified in $G$. In particular, if $\lambda$ is fully ramified in $G$ then the nontrivial Sylow subgroups of $G/Z$ are not cyclic.
\end{lem}
\begin{proof}
This is \cite[Theorem 8.3]{N18} and \cite[Problem 8.1]{N18}.
\end{proof}

The following results, which extend the Howlett--Isaacs theorem, appeared in \cite{higgs} in the language of projective representations. 

\begin{thm}[Higgs]\label{lem:demeyerjanusz2}
Let $Z\normal G$ and let $\lambda\in{\rm Irr}(Z)$ be $G$-invariant and suppose that $|{\rm Irr}(G|\lambda)|=2$. Then $G/Z$ is solvable. Furthermore, if $q$ is an odd prime and $Q/Z\in{\rm Syl}_q(G/Z)$, then $|{\rm Irr}(Q|\lambda)|=1$, and in particular, $|Q/Z|$ is a square. If $q=2$, $|{\rm Irr}(Q|\lambda)|=2$, and in particular, $|Q/Z|$ is an odd power of 2. 
\end{thm}

\begin{proof}
This is \cite[Theorem B and Proposition 2.2]{higgs}.
\end{proof}

 It remains open whether if $N\normal G$ and $\lambda\in\Irr(N)$ is $G$-invariant then $|\Irr(G|\lambda)|=3$ implies that $G/N$ is solvable. This solvability would simplify our proofs for the case of rank 3 and would be helpful when dealing with larger values of $k(B)$ in future investigations.  We mention that the group $\SL(2,5)$ shows that this solvability can no longer be guaranteed if $|\Irr(G|\lambda)|\geq 4$.

\section{Proof of the main theorem}\label{sec:proof}

In this section, we prove Theorem \ref{thm:reduced}, thus concluding the proof of Theorem \ref{thm:A}. For the remainder of the paper we will write $|D|=p^d$.

\subsection{Preliminaries on primitive permutation groups}

To prove Theorem \ref{thm:reduced} we need to deal with primitive permutation groups. The degree of a primitive permutation group is the size of the set on which it is acting, in our case, $|D|$. We begin by recalling some important definitions related to the structure of these groups.

Let $V$ be a $p$-elementary abelian group. Then $V$ can be identified with the Galois field $\mathbb{F}_{p^k}$. The semilinear group $\Gamma(V)$ is defined as the set of maps
$$\Gamma(V)=\{x\mapsto ax^\sigma\mid x\in V, a\in V\setminus \{0\},\sigma\in\Gal(\mathbb{F}_{p^k}/\mathbb{F}_p)\}.$$
Notice that $\Gamma(V)$ is a metacyclic group. Indeed, it contains the normal subgroup
$$\Gamma_0(V)=\{x\mapsto ax\mid a\in V\setminus \{0\}\}$$
which is cyclic and isomorphic to the multiplicative group of $V$, and $$\Gamma(V)/\Gamma_0(V)\cong\Gal(\mathbb{F}_{p^k}/\mathbb{F}_p)$$ is cyclic of order $k$.
We frequently use the notation $\Gamma(p^k)=\Gamma(V)$. If a subgroup $H\leq \Gamma(V)$ is $2$-transitive on $V$ then $\Gamma_0(V)\sbs H$. Semilinear groups appear very naturally in the classifications of primitive permutation groups.

%(mentioned in p. 4 of the arXiv preprint of Hung--Sambale--Tiep).

The following are the classifications of affine primitive permutation groups that we will need. We eliminate any case that contradicts any hypothesis in Theorem \ref{thm:reduced}, which we recall here for the reader's convenience.

\begin{hyp}\label{hypo} Suppose that $G$ is a finite group and $p\geq 5$ is a prime number. Suppose that $Z=\zent G=\oh{p'}G$ and that $D$ is a normal Sylow $p$-subgroup of $G$. Let $K$ be a $p$-complement of $G$, so $G=D\rtimes K$. Assume that $D$ is a minimal normal subgroup of $G$ and that $G/Z$ is an affine primitive permutation group of rank 2,3 or 4 on $D$.
\end{hyp}

%Recall that in our case, $p\geq 5$, $G/Z$ is $p$-solvable and $DZ/Z$ is a Sylow $p$-subgroup of $G/Z$, which eliminates many cases of the classifications.

We start by stating the classification of the $p$-solvable primitive permutation groups of rank 2 (that is, $2$-transitive permutation groups) which is due to D. Passman. 

\begin{thm}[Passman]\label{thm:passman}
Let $G/Z$ be an affine primitive permutation group of rank $2$ on $D$. Then
\begin{enumerate}
\item $K/Z\leq \Gamma(p^d)$,
\item $d=2$ and the possible structures of $K/Z$ are shown in \cite[Table 15.1]{sambale}. 
\end{enumerate}
\end{thm}
\begin{proof}
See \cite{passman}.
\end{proof}

As pointed out by B. Sambale, the primitive group missing in \cite[Table 15.1]{sambale} is PrimitiveGroup($59^2, 84$) in \cite{GAP}.

The primitive permutation groups of rank 3 were classified by M. Liebeck. Again we eliminate every possibility contradicting any hypothesis from Theorem \ref{thm:reduced}.

\begin{thm}[Liebeck]\label{thm:liebeck}
Under Hypotheses \ref{hypo}, assume that the rank of $G/Z$ on $D$ is $3$.  Then
\begin{enumerate}
\item $K/Z\leq\Gamma(p^d)$,
\item $d=2$ and if $K/Z$ is solvable then $p\in\{13, 17, 19, 29, 31, 47\}$,
\item $p^d\in\{5^4, 7^4\}$, or
\item there is an imprimitivity decomposition on $2$ subspaces $D=V_1\times V_2$, and each $\norm K {V_i}/ \cent K {V_i}$ is $2$-transitive on $V_i$ and so it is classified in Theorem \ref{thm:passman},
\end{enumerate}
\end{thm}
\begin{proof}
This is the content of \cite{liebeck}. Notice that cases (3)--(11) of part (A) of the main result are not $p$-solvable. The same argument excludes many cases in part (C) of the main result. The solvable cases in part (ii) appear in \cite[Theorem 1.1]{foulser}.
\end{proof}

When the rank of $G/Z$ on $D$ is $4$ then we have that $l(B)\leq 2$ using Brauer's formula from Lemma \ref{lem:brauer formula}. By Lemma \ref{lem:l(B)} we have $|\Irr(K|\lambda)|=2$ and by the main result of \cite{higgs}, $K/Z$ (and then $G$) is solvable. Thus we make use of the following classification due to D. A. Foulser.

\begin{thm}[Foulser]\label{thm:foulser}
Under Hypotheses \ref{hypo}, assume that the rank of $G/Z$ on $D$ is $4$.  Then
\begin{enumerate}
\item $K/Z\leq\Gamma(p^d)$,
\item $d=2$ and $5\leq p\leq 71$.% $p^d\in\{11^2,13^2, 17^2, 19^2, 23^2, 29^2, 31^2, 37^2,41^2,43^2, 47^2,53^2,59^2,61^2,67^2,71^2\}$.
\item $p^d\in\{5^4, 7^3, 7^4\}$, or
\item there is an imprimitivity decomposition on $2$ or $3$ subspaces (so $D=V_1\times V_2$, or $D=V_1\times V_2\times V_3$) and each $\norm K {V_i}/ \cent K {V_i}$ is $2$-transitive on $V_i$ and so it is classified in Theorem  \ref{thm:passman}.
\end{enumerate}
\end{thm}
\begin{proof}
See \cite[Theorem 1.2]{foulser}.
\end{proof}

In cases (ii) of Theorems \ref{thm:liebeck} and \ref{thm:foulser}, the structure of $K/Z$ is determined in \cite[Lemma 2.8]{mw}. This will be further discussed in Section \ref{sec:imprimitive}.

The remaining part of this paper will consist in verifying Theorem \ref{thm:reduced} for all possible structures in these classifications. In essence, we find contradictions unless $p^d\in\{5, 7\}$, which is only possible when $K/Z\leq\Gamma(p^d)$. For the remainder of the paper, we assume the hypotheses from Theorem \ref{thm:reduced}.

\subsection{The case $|D|=p^2$} In this section we prove Theorem \ref{thm:reduced} when $|D|=p^2$ except when $K/Z\leq \Gamma(p^2)$. We will deal with this case in Section \ref{sec:semilinear}.

\begin{pro}\label{pro:p2} Assume hypotheses of Theorem \ref{thm:reduced}. If $d=2$, then there is some $1\neq u\in D$ such that some $b\in{\rm Bl}(\cent G u)$ inducing $B$ satisfies $l(b)>1$.\end{pro}

\begin{proof}
Assume by way of contradiction that every $1\neq u\in D$ and $b_u\in \Bl(\cent G u)$ such that $b_u^G=B$ satisfies $l(b_u)=1$.  Write $C=\cent G u$. Notice that since $D\normal G$ and $D\sbs C$ we have that $D$ is a defect group of $b_u$ by \cite[Theorem 4.8]{N98}.

We claim first that $C=\cent G D$. By \cite[Theorem 9.10]{N98}, $b_u$ dominates a unique block $\overline{b}_u\in\Bl(C/\langle u \rangle)$ with defect group $D/\langle u \rangle$, which is cyclic since $d=2$. Notice that if $c\in\Bl(\cent G D)$ is a root of $b_u$, so that $c^C=b_u$, then $c^G=b_u^G=B$ by \cite[Problem 4.2]{N98}. Hence $c$ lies under $B$ by \cite[Lemma 9.8]{N98}, then $c=b$, where $b$ is the block of $\cent G D$ containing $1_D\times\lambda$ (see the comment before Step 5 of Theorem \ref{thm:reduction}). Then $\overline{b}$, the unique block of $\cent
G D/\langle u\rangle=\cent{C/\langle u\rangle}{D/\langle u \rangle}$ dominated by $b$, is a root of $\overline{b_u}$. Since $b$ is $C$-invariant, $\overline{b}$ is $C/\langle u \rangle$-invariant. By \cite[Theorem 9.10]{N98} we have $l(\overline{b}_u)=l(b_u)=1$. By \cite[Theorem 11.13]{N98}, $\overline{b}_u$ has inertial index 1, that is, $C/\langle u \rangle=\cent G D/\langle u \rangle$. Hence $C=\cent G D$, as desired.

Then $C=Z\times D$. In particular, the action of $G/ZD$ is Frobenius on $D$. This implies that the Sylow subgroups of $G/ZD$ are cyclic or generalized quaternion, by \cite[Theorem 6.10]{I08} and \cite[Theorem 6.11]{I08}, so they have trivial Schur multiplier. By Lemma \ref{lem:schur}, $\hat\lambda$ extends to $G$ and so does $\lambda$. Then $5=|\Irr(G|\lambda)|=|\Irr(G/Z)|$ by Gallagher's theorem. However there is no finite group with 5 conjugacy classes with normal Sylow $p$-subgroups of order $p^2$ with $p\neq 2,3$.
\end{proof}

The next propositions prove Theorem \ref{thm:reduced} when $|D|=p^2$ as long as $K/Z$ is not a subgroup of the semilinear group. That case will be done in Section \ref{sec:semilinear}.

\begin{pro}\label{pro:d=2 rank 2}
Assume that the rank of $G/Z$ on $D$ is $2$ and that $K/Z$ is not isomorphic to a subgroup of the semilinear group. Then Theorem \ref{thm:reduced} holds.
\end{pro}
\begin{proof}
In this case, $K/Z$ is one of the groups listed in (ii) of Theorem \ref{thm:passman}, so they can be found in \cite[Table 15.1]{sambale}. Also, by Lemma \ref{lem:l(B)} we have $|\Irr(K|\lambda)|\leq 4$.

Assume first that there is some $Z\sbs N\normal K$ with $N/Z\cong \SL(2,5)$. Then since the Sylow subgroups of $\SL(2,5)$ have trivial Schur multiplier, $\lambda$ extends to $N$, and by Gallagher's theorem there is a degree-preserving bijection $\Irr(N/Z)\rightarrow \Irr(N|\lambda)$. Thus $\Irr(N|\lambda)$ contains characters of 6 distinct degrees, so there are at least 6 $K$-orbits in $\Irr(N|\lambda)$, which would yield at least 6 distinct characters in $\Irr(K|\lambda)$, a contradiction.

Thus we are in the remaining cases, where there is some $Z\sbs N\normal K$ with $N/Z\cong\SL(2,3)$ and $K/N$ is either cyclic or $|K:N|=2q$ for some odd prime $q$. Again $\SL(2,3)$ has Sylow subgroups with trivial Schur multiplier, so $\lambda$ extends to $N$ and by Gallagher's theorem, $\Irr(N|\lambda)=\{\delta_1,\delta_2,\delta_3,\gamma_1,\gamma_2,\gamma_3,\eta\}$, where $\delta_i(1)=1$, $\gamma_i(1)=2$, $\eta(1)=3$. In particular, if $K=N$ we have a contradiction. By uniqueness, $\eta$ is $K$-invariant. Also, since $|\Irr(K|\lambda)|\leq 4$ there are at most 4 $K$-orbits on $\Irr(N|\lambda)$, forcing $|\Irr(K|\eta)|\leq 2$. If $K/N$ is cyclic then $\eta$ extends to $K$, which forces $|K/N|=2$ (because $N<K$). However, then some $\delta_i$ is also $K$-invariant and extends to $K$, which implies $|\Irr(K|\lambda)|>|\Irr(K|\delta_i)|+|\Irr(K|\eta)|=4$, a contradiction. Thus we assume $|K:N|=2q$ where $q$ is some odd prime. However, since $|\Irr(K|\lambda)|\leq 2$ we know $\eta$ is fully ramified in $Q$, where $Q/N\in\Syl_q(K/N)$, by Lemma \ref{lem:demeyerjanusz} and Theorem \ref{lem:demeyerjanusz2}. This contradicts the fact that $Q/N$ is cyclic, and we are done.
\end{proof}

\begin{pro}\label{pro:d=2}
Assume that the rank of $G/Z$ on $D$ is 3 or 4, $d=2$ and that $K/Z$ is not isomorphic to a subgroup of the semilinear group. Then Theorem \ref{thm:reduced} holds.
\end{pro}
\begin{proof}
First, we may assume $l(B)>1$ using Theorem \ref{thm:kb small}. By Proposition \ref{pro:p2}, there is some $p$-element $1\neq u\in D$ and $b\in \Bl(\cent G u)$ inducing $B$ such that $l(b)>1$. By Brauer's formula (see Lemma \ref{lem:brauer formula}), it follows that the rank of $G/Z$ on $D$ is exactly 3 and $l(B)=2$. Then $|{\rm Irr}(K|\lambda)|=2$ by Lemma \ref{lem:l(B)} and by Theorem \ref{lem:demeyerjanusz2} $K/Z$ is solvable, and it must be one of the solvable groups in cases (ii) and (iv) of Theorem \ref{thm:liebeck}. Furtheremore, $|K:Z|_2$ is an odd power of 2 and $|K:Z|_q$ is a square for every prime $q\neq 2$.

Let $u$, $v$ be representatives of the nontrivial $G$-conjugacy classes on $D$, and assume the $B$-subsections $(u, b_u)$ and $(v, b_v)$ satisfy $l(b_u)=2, l(b_v)=1$. Then as in the first two paragraphs of the proof of Proposition \ref{pro:p2} we have $|\cent G u:\cent G D|=2$ and $|\cent G v:\cent G D|=1$. Since $ZD=\cent G D$ we have
\begin{equation}\label{eq:p2}p^2-1=|D|-1=|G:\cent G u|+|G:\cent G v|=\frac{3}{2}|G/ZD|=\frac{3}{2}|K:Z|.\end{equation}
If $K/Z$ is as in case (ii) of Theorem \ref{thm:liebeck}, we have that $p\in\{13, 17, 19, 29, 31, 47\}$. By equation \ref{eq:p2} there is some odd prime $q$ such that $K/Z$ has Sylow $q$-subgroups of order $q$. This is a contradiction.

%If $7<p\leq 71$, the expression \ref{eq:p2} for these primes shows that there is always an odd prime $q$ such that $|K:Z|_q=q$, contradiction since $|K:Z|_q$ is an even power of $q$. If $p=5$, we obtain that $|K:Z|=16$, contradiction since $|K:Z|_2$ is an odd power of 2. If $p=7$, we obtain that $|K:Z|=32$, and then $|G:Z|=2^5\cdot 7^2$. The unique primitive permutation group of this order is ${\rm PrimitiveGroup}(7^2,19)$. In this group there is $C/Z\lhd G/Z$ cyclic of order 16. Then $\lambda$ extends to $C$ and $|{\rm Irr}(C|\lambda)|=16$, contradictión since $|K:C|=2$ and $|{\rm Irr}(K|\lambda)|=2$.

% , except when $p=5,7$. Suppose first that $p>7$, then $\lambda$ extends to $Q$, but by Lemma \ref{lem:l(B)}, $|\Irr(K|\lambda)|=2$, contradicting Lemma \ref{lem:demeyerjanusz2}. If $p=5$, $|K:Z|=16$

If $K/Z$ is as in case (iv) of Theorem \ref{thm:liebeck} then each space of imprimitivity has size $p$ and we can write $D=V_1\times V_2$. Then $K/Z$ has a subgroup $H/Z$ of index 2 such that $H/Z=A/Z\times B/Z$ and $A/Z$ is a solvable doubly-transitive group on $V_1$ (and $B/Z$ on $V_2$). Hence $A/Z$ is classified in Theorem \ref{thm:passman} and is a solvable $2$-transitive group of degree $p$. This is only possible when $A/Z = \Gamma(p)$, and then $|K:Z|=2(p-1)^2$, which contradicts Equation (\ref{eq:p2}) using that $p\geq 5$. \end{proof}

\subsection{Semilinear groups}\label{sec:semilinear} In this section we prove Theorem \ref{thm:reduced} for case (i) in Theorems \ref{thm:passman}, \ref{thm:liebeck} and \ref{thm:foulser}.

\begin{pro}\label{pro:semilinear}
Under the hypotheses of Theorem \ref{thm:reduced}, if $K/Z$ is isomorphic to a subgroup of $\Gamma(p^d)$ then either $p=5$ and $d=1,2$ or $p=7$ and $d=1$. Further, if $x$ is some nontrivial $p$-element of $G$ and $p^d=5^2$ then $l(b)=1$ where $b\in\Bl(\cent G x)$ induces $B$. In particular, Theorem \ref{thm:reduced} holds.
\end{pro}
\begin{proof}
In this situation there is a subgroup $H$ of $K$ containing $Z$ with $H/Z$ cyclic of order $s\, |\, p^d-1$ and index $|G:H|=t\, |\, d$.  Set $l=l(B)$. We have that $2\leq l\leq 4$. 
We have that $l=|\Irr (K |\lambda)|$ by Lemma \ref{lem:l(B)}.  On one hand we have that $|\Irr( H|\lambda)|= |H/Z|=s$ and on the other hand, by Clifford's theorem,  $|\Irr( H|\lambda)|\leq lt $. So $s\leq lt\leq ld$.  

Now, $\bar K=K/Z$ acts on $D\setminus\{1_D\}$ faithfully by conjugation. Let $m$ be the number of orbits of this action. We have 
$$p^d-1\leq m|K/Z| =mst\leq mld^2$$
and it follows that $p\in\{5, 7\}$ and $d\in\{1, 2\}$. If $p^d=7^2$ then $48\leq 4ml$ and since $m\leq 3$ this forces $(m,l)=(3,4)$.  Let $\{x_1,\dots,x_m\}$ be representatives of the nontrivial $G$-orbits of the action by conjugation on $D$ (these coincide with the representatives of the $K/Z$-orbits). Using Brauer's formula from Lemma \ref{lem:brauer formula}, $$k(B)=l(B)+\sum_{i=1}^m\sum_{b\in\Bl(\cent G x_i),\\b^G=B} l(b)\geq l+m$$
and then $5=k(B)\geq m+l$, a contradiction.

Now notice that if $p^d=5^2$ we have $24\leq 4ml$ and then $6\leq ml$. Arguing as before we get $5=k(B)\geq m+l$ and since $m\leq 3$ it follows that $(m,l)\in\{(2,3),(3,2)\}$. In any case, $k(B)=m+l$ and the last part follows.

It follows that Theorem \ref{thm:reduced} holds, unless possibly if $p^d=5^2$ but in this case the fact that $l(b)=1$ contradicts Proposition \ref{pro:p2}.
\end{proof}

%\begin{cor}\label{cor:semilinear}
%If $K/Z\leq \Gamma(p^d)$ then Theorem \ref{thm:reduced} holds.
%\end{cor}
%\begin{proof}
%By Proposition \ref{pro:semilinear}, the only case that could possibly fail is $p^d=5^2$ but the fact that $l(b)=1$ contradicts Proposition \ref{pro:p2}.
%\end{proof}

\subsection{The exceptions}\label{sec:exceptions} In this section we deal with the case (iii) in Theorems \ref{thm:liebeck} and \ref{thm:foulser}. We remark that in all cases, $|V|<2500$ so \cite{roney-dougal} guarantees that the \cite{GAP} library of such groups is complete.

We begin with a general result which will help us in proving Theorem \ref{thm:reduced} in this and the next section. 

\begin{lem}\label{lem:contradiction}
Assume hypotheses of Theorem \ref{thm:reduced}, and that the rank of $G/Z$ on $D$ is $3$ or $4$. Then there is $\mu\in\Irr(D)$ such that $\mu\times\lambda\in\Irr(D\times Z)$ is fully ramified in its stabilizer $G_\mu$. In particular, $|G_\mu/(D\times Z)|$ is a square. If the rank is $4$ then this happens for all nontrivial $\mu\in\Irr(D)$. If the rank is $3$, $|G_\mu/(D\times Z)|_q$ is a square for all odd primes $q$  and for all nontrivial $\mu\in\Irr(D)$.
\end{lem}
\begin{proof} 
By using Brauer's formula from Lemma \ref{lem:brauer formula} we have $l(B)\in\{2, 3\}$. By Lemma \ref{lem:l(B)}, $l(B) = |\Irr(K|\lambda)| = |\Irr(G|1_D \times \lambda)|$, and it follows that  the number of characters in the block $B$ that do not have $D$ in their kernel is $n = k(B) - l(B) \in \{2, 3\}$.  If $n = 2$, then $G/Z$ has rank $3$ on $D$, and therefore there exist two nonconjugate characters $\mu_i \in \irr D$ for $i = 1, 2$ such that $\mu_i \times \lambda$ are fully ramified in their stabilizer. Hence, $|G_{\mu_i \times \lambda}/(D\times Z)| = |G_{\mu_i}/(D\times Z)|$ is a square.

In the second situation, when $n=3$ and $G/Z$ has rank $3$ over $D$, we observe that one of the nontrivial characters $\mu_1$ of $D$ satisfies the condition that $\mu_1\times \lambda$ is fully ramified in its stabilizer. Additionally, in this case, we have that $|G_{\mu_1}/(D\times Z)|$ is a square. On the other hand, if $\mu_2\in\Irr(D)$ is a non-trivial character that is not $G$-conjugate to $\mu_1$, then $|\irr{G_{\mu_2}|\mu_2\times\lambda}|=2$. By Theorem \ref{lem:demeyerjanusz2}, we conclude that $|G_{\mu_1}/(D\times Z)|_q$ is a square for every odd prime $q$.

%when $n = 3$ and $G/Z$ has rank 3 on $D$, there exists a unique $\mu \in \irr D$ such that $\mu \times \lambda$ is fully ramified in its stabilizer. Therefore, in this case as well, it follows that $|G_{\mu} /DZ|$ is a square.

Finally, if $n = 3$ and the rank of $G/Z$ on $D$ is 4, there exist three irreducible characters $\mu_i$  with $i=1,2,3$ such that $|G_{\mu_i \times \lambda}/(D\times Z)| = |G_{\mu_i}/(D\times Z)|$ is a square.
\end{proof}

\begin{pro}\label{pro:liebeckex}
Assume that $K/Z$ is as in case (iii) of Theorem \ref{thm:liebeck}. Then Theorem \ref{thm:reduced} holds.
\end{pro}
\begin{proof}
\textit{The case $p^d=5^4$:}
The list of groups verifying this checked in GAP is formed by the groups $\mathrm{PrimitiveGroup}(5^4,i)$ where 
\begin{align*}
i\in\{&207, 210, 218, 219, 220, 221, 222, 223, 224, 225, 328, 329, 330, 331, 332, 333,
 334, 410,\\& 411, 412, 413, 417, 418, 419, 434, 435, 436, 446, 487, 490, 492, 494, 496, 503, 520\}.
 \end{align*}

 %\textcolor{violet}{If we consider $i\in \{ 207, 210, 218, 19, 220, 221, 222, 223, 224, 225, 328, 329, 330,331,333, \\ 
 %334, 410, 411, 412, 413, 417, 418, 419, 434, 435, 436, 446\}$ we can see that there exists    a nontrivial $\mu \in \irr D$ such that $|G_\mu/(D\times Z)|_3=3$ which leads to a contradiction.}
 
We observe that if $i\not\in\{487,492, 494, 496, 503\}$ there is $1_D\not=\mu \in \irr D$ such that $|G_\mu/(D\times Z)|_3=3$ a contradiction with Lemma \ref{lem:contradiction}.
  
 If $i\in \{487,496\}$, there are nontrivial and not $G$-conjugate $\mu_1,\mu_2\in{\rm Irr}(D)$ such that $|G_{\mu_1}/D\times Z|=2=|G_{\mu_2}/D\times Z|$, a contradiction since at least one of these has to be a square by Lemma \ref{lem:contradiction}.  If $i\in \{492, 494,503\}$, then $\lambda$  extends to $K$ (since every Sylow subgroup of $K/Z$ is cyclic, using Lemma \ref{lem:cyclic} and \cite[Theorem 5.10]{N18}) and by Gallagher's theorem $l(B)=|\irr {K\,|\, \lambda}|=|{\rm Irr}(K/Z)|>3$, a contradiction.  In the remaining cases $\lambda$ extends to  $F$ where $F/Z=\fit {K/Z}$ is a cyclic group of order strictly greater than 12. By Gallagher's theorem, $|\Irr(F| \lambda)|=|\Irr(F/Z)|=|F/Z|>12$ and $|K:F|=4$ so it follows that $l(B)=|\Irr(K| \lambda)| >3$, again a contradiction.
 
\textit{Case $p^d=7^4$:}
 The list of groups verifying this checked in GAP is formed by the groups $\mathrm{PrimitiveGroup}(7^4,i)$ where 
\begin{align*}
i\in\{&721, 723, 774, 775, 783, 785, 789, 797, 800, 802, 805, 843, 844, 847, 849, 850, 851, \\&852, 855, 869, 873, 876, 877, 878,879, 880, 885, 886, 890, 891, 894, 962, 964,965,\\& 992, 1015, 1039, 1049 \}.\end{align*}

If $i\in \{721, 723\}$ we have that $\lambda$ extends to $K$ and
     $|\irr{K\, |\, \lambda}|>3$, a contradiction. If $i\in \{774,855\}$,   $\lambda$ extends to  $F$ where $ F/Z=\fit {K/Z}$ is a cyclic group of order strictly larger than 12. By Gallagher's theorem $|\Irr(F|\lambda)|=|\Irr(F/Z)|=|F/Z|>12$. Since $|K:F|=4$ it follows that $|\Irr(K | \lambda)| >3$, a contradiction.  If $i\in \{  802, 805, 965 \}$, we can find non $G$-conjugate character $\mu_1, \mu_2\in\Irr(D)$ such that $|\Irr(G|\mu_k\times\lambda)|>1$ for $k=1,2$ and this forces $|\Irr(G|1_D\times\lambda)|=|\Irr(K|\lambda)|=1$, a contradiction. 
     
   %  Now, if $i=775$, we have that there exists a non-trivial character $\mu \in \irr D$ such that $|G_\mu/D\times Z|$ is not a square, and from this, it follows that $l(B)$ must be $2$, since otherwise $\mu\times\lambda$ would be fully ramified in $G_\mu$. On the other hand, we observe that the group $K/Z$ has a Sylow $3$ subgroup $P/Z$ of central order $3$. By Lemma \ref{lem:cyclic}, the character $\lambda$ extends to $P$. Let $\tilde\lambda$ be one of the extensions. By Gallagher's theorem, the other extensions are of the form $\tilde\lambda\gamma$ and $\tilde\lambda\gamma^2$, where $\gamma$ is a nontrivial irreducible character of $P/Z$. Since $3$ does not divide $K/P$, at least one of the extensions of $\lambda$ to $P$, denoted as $\xi$, is $K$-invariant. In fact, all three extensions are $K$-invariant, as the other two extensions are $\xi\gamma$ and $\xi\gamma^2$ (both $\gamma$ and $\gamma^2$ are $K$-invariant).  Hence, we conclude that $l(B)$ must be at least 3, which leads to a contradiction.   
     
Now, if $i=775$ then $|K/Z|_3=3$ and $|K/Z|_5=5$. Then, using Lemma \ref{lem:l(B)} and Theorem \ref{lem:demeyerjanusz2} we have that $l(B)=|\Irr(K|\lambda)|=3$. If $P/Z\in\Syl_3(K/Z)$, since $P/Z$ is cyclic of order 3 we have that by Lemma \ref{lem:cyclic} $\lambda$ extends to $P$ and $|\Irr(P|\lambda)|=3$ by Gallagher's theorem. Furthermore, $P/Z$ is central in $K/Z$. Since $|K/P|$ is not divisible by $3$, there exists some $K$-invariant $\xi\in\Irr(P|\lambda)$.  By Clifford theory, $|\Irr(K|\xi)|\leq 2$ and therefore, by Lemma \ref{lem:demeyerjanusz} and Theorem \ref{lem:demeyerjanusz2}, $|K/P|_5=|K/Z|_5$ is a square, a contradiction.

    % \textcolor{violet}{ Now, if $i = 775$, we may assume that $l(B) = 3$ since otherwise $|G/ZD|$ would be a square. On the other hand, we observe that the group $K/Z$ has a Sylow $3$-subgroup $P/Z$ which is  central of order $3$. By Lemma \ref{lem:cyclic}, the character $\lambda$ extends to $P$. Let $\tilde{\lambda}$ be one of the extensions. According to Gallagher's theorem, the other extensions are of the form $\tilde{\lambda}\gamma$ and $\tilde{\lambda}\gamma^2$, where $\gamma$ is a nontrivial irreducible character of $P/Z$. Since $3$ does not divide $K/P$, at least one of the extensions of $\lambda$ to $P$, denoted as $\xi$, is $K$-invariant. In fact, all three extensions are $K$-invariant, as the other two extensions are $\xi\gamma$ and $\xi\gamma^2$ (both $\gamma$ and $\gamma^2$ are $K$-invariant). This implies that every extension is fully ramified in $K$, but this is not possible because $K/P$ is not a square ($|K/P|_5=5$).}

 In the remaining cases, there is a nontrivial $\mu \in \irr D$ such that $|G_\mu/(D\times Z)|_3$  is not square, a contradiction with Lemma \ref{lem:contradiction}.   \end{proof}

\begin{pro}\label{pro:foulserex}
Assume that $K/Z$ is as in case (iii) of Theorem \ref{thm:foulser}. Then Theorem \ref{thm:reduced} holds.
\end{pro}
\begin{proof}

 \textit{The case $p^d=5^4$.} 
The list of groups verifying this checked in GAP is formed by the groups $\mathrm{PrimitiveGroup}(5^4,i)$ where
\begin{align*}
i\in\{&209, 216, 217, 284, 285, 286, 287, 288, 289, 290, 291, 292, 293, 323, 324, 325, 326, \\&327, 337, 338, 339, 340, 341,
  343, 345, 404, 414, 416, 424, 425, 426, 442, 443, 444,\\ &445, 460, 461, 463, 466, 484, 511, 518, 519, 527, 528, 529,
  532, 533, 538, 541\}.
  \end{align*}
  Since $l(B)=2$ then $|\Irr(G|\lambda\times 1_D)|=2$ and by Lemma \ref{lem:demeyerjanusz}, $\lambda\times 1_D$ is fully ramified in $Q$ for any $Q/(D\times Z)\in\Syl_q(G/(D\times Z))$ and any odd prime $q$. In particular, $G/(D\times Z)$ can not have nontrivial cyclic Sylow $q$-subgroups for any odd $q$ by Lemma \ref{lem:cyclic}. This excludes every case above except for
$$i\in\{216, 217, 288, 289, 290, 291, 292, 293, 323, 324, 337, 338, 460, 461, 484, 533\}.$$
For all these groups it is possible to find some nontrivial $\mu\in \irr D$ such that $|G_\mu/(D\times Z)|_3=3$, contradicting Lemma \ref{lem:contradiction}.

 \textit{Case $p^d=7^4$:}
 The list of groups verifying this checked in GAP is formed by the groups $\mathrm{PrimitiveGroup}(7^4,i)$ where 
 \begin{align*}
 i\in\{& 634, 638, 665, 666, 674, 675, 691, 709, 710, 736, 737, 738, 740, 755, 757, 782, 786,\\& 787, 788, 790, 793, 803, 820,
  822, 823, 825, 828, 845, 846, 848, 867, 868, 870, 871,\\& 875, 892\}.  \end{align*} If $i\in\{634, 638\}$ we have that $\lambda$ extends to $K$  and hence $|\irr{K\,|\,\lambda}|>3$, a contradiction. In the remaining cases,  there is a nontrivial $\mu \in \irr D$ such that $|G_\mu/(D\times Z)|$  is not square, a contradiction with Lemma \ref{lem:contradiction}.

 \textit{Case $p^d=7^3$:}
The list of groups verifying this checked in GAP is formed by the groups $\mathrm{PrimitiveGroup}(7^4,i)$ where 
$$i\in\{45, 60, 61, 70, 72, 75,78, 84,85\}.$$ If $i=45$  we have that $\lambda$ extends to $K$  and hence $|\irr{K\,|\,\lambda}|>3$, a contradiction.  In the remaining cases,  there is a nontrivial $\mu \in \irr D$ such that $|G_\mu/(D\times Z)|$  is not square, a contradiction with Lemma \ref{lem:contradiction}.
\end{proof}

\subsection{The imprimitive case}\label{sec:imprimitive}

In this section we prove Theorem \ref{thm:reduced} for case (iv) in Theorem \ref{thm:liebeck} and \ref{thm:foulser}.

In this case, we can write the Sylow $p$-subgroup of $G$ as $D = V_1 \times V_2$ or as $D = V_1 \times V_2 \times V_3$. In this context, the subgroups $V_i$ are sometimes called  imprimitivity spaces. Let $A/Z=\norm {{K/Z}}{V_1}/\cent{K/Z}{V_1}$. Then $K/Z\cong (A/Z)\wr S$ where $S\leq \SSS_n$, and $n$ is the number of imprimitivity spaces (this can be found in \cite[Theorem 2.8]{mw}). Recall that we write $|D|=p^d$. Let $|V_1|=p^m$ so that $d=2m$ or $d=3m$, and let $H/Z$ be the base group of $(A/Z)\wr S$, which is a direct product of $n$ copies of $A/Z$, so that $H\normal K$ and $K/H\cong S$. We assume this notation and structure for this whole section.

The following result deals with the imprimitive case in rank 3.

\begin{pro}\label{pro:imp3}
Assume that $K/Z$ is as in case (iv) of Theorem \ref{thm:liebeck}. Then Theorem \ref{thm:reduced} holds.
\end{pro}
\begin{proof}
In this case, $K/Z=(A/Z)\wr \SSS_2$. Write $|V_1|=p^m$ so that $d=2m$ (and recall that $d>2$ by Proposition \ref{pro:d=2}). The orbits of $K/Z$ on $D$ have sizes $1$, $2(p^m-1)$ and $(p^m-1)^2$, by \cite[Table 12]{liebeck}. By \cite[Theorem 18.10]{Hup98}, the orbits of the action of $K/Z$ on $D$ have exactly the same sizes as the $K$-orbits on $\irr D$. Notice that if $\alpha, \beta$ are representatives of the nontrivial $K$-orbits on $D$ then at least one of $\alpha$ and $\beta$ are fully ramified in their stabilizer in $G$. 

Assume first that $\alpha$ is fully ramified in $G_\alpha$ so that $|G_\alpha:ZD|$ is a square. We have $|G:G_\alpha|=2(p^m-1)$, $|G_\alpha:ZD|=s^2$ so $|G:ZD|=2s^2(p^m-1)$. Now $|G:ZD|=|K:Z|=2|A/Z|^2$,  which forces $p^m-1$ to be a square. This is a contradiction with \cite[p. 301]{mordell} if $m$ has an odd prime divisor, so we assume that $m$ is a $2$-power, so $p^m-1=(p^{\frac{m}{2}}+1)(p^{\frac{m}{2}}-1)$ which can not be a square.

Thus we have that $\beta$ is fully ramified in $G_\beta$. We have $|G:ZD|=2|A:Z|^2$ is not a square. Since $|G:G_\beta|=(p^d-1)^2$ and $|G_\beta :ZD|=|G:ZD|/|G:G_\beta|$ can not be a square, a contradiction since $\beta$ is fully ramified in $G_\beta$.
\end{proof}

In what follows we will be dealing with imprimitive groups of rank 4, that is, those in case (iv) of Theorem \ref{thm:foulser}. In this case, we may have $2$ or $3$ imprimitivity spaces, and in these cases respectively $S\cong \SSS_2$ or $S\cong \CCC_3, \SSS_3$. Recall that in this case, Brauer's formula forces $|\Irr(K|\lambda)|=l(B)\leq 2$ and that we assumed $l(B)>1$ via Theorem \ref{thm:kb small}. Thus by Theorem \ref{lem:demeyerjanusz2}, $G$ is solvable. Then the possible structure of $A/Z$ is determined by Huppert's theorem, i.e., $A/Z$ is one of the solvable groups from Theorem \ref{thm:passman}. We deal with each of these cases separately.

%\begin{pro}\label{pro:imp4-1}
%Assume that $S\cong\SSS_2$ and that $A/Z\leq \Gamma(p^m)$. Then  $|\Irr(K|\lambda)|>2$. In particular, Theorem \ref{thm:reduced} holds.
%\end{pro}
%\begin{proof}
%By \cite[Lemma 4]{MRS23}, $K$ has an abelian Sylow $q$-subgroup $Q$ for some odd prime $q$, and $|QZ/Z|>q$. Thus $\lambda$ extends to $QZ$ and $QZ/Z\in\Syl_q(G)$, and by Lemma \ref{lem:demeyerjanusz2} we have $|\Irr(K|\lambda)|>2$.
%\end{proof}

\begin{pro}\label{pro:imp4}
Assume that $S\cong\SSS_2$. Then  $|\Irr(K|\lambda)|>2$. In particular, Theorem \ref{thm:reduced} holds.
\end{pro}
\begin{proof}
Assume first that $A/Z\leq \Gamma(p^m)$. By \cite[Lemma 4]{MRS23}, $K$ has an abelian Sylow $q$-subgroup $Q$ for some odd prime $q$, and $|QZ/Z|>1$. Thus $\lambda$ extends to $QZ$ and $QZ/Z\in\Syl_q(K/Z)$, and by Theorem \ref{lem:demeyerjanusz2} we have $|\Irr(K|\lambda)|>2$, as desired.

Hence we may assume $A/Z$ is one of the solvable groups in case (ii) of Theorem \ref{thm:passman}. It is easy to check in \cite{GAP} that for all possibilities of $A/Z$, there is a normal subgroup $N/Z\normal A/Z$ such that $N/Z\cong \SL(2,3)$. We assume first that $S\cong \SSS_2$. Let $B/Z$ denote the $K$-conjugate subgroup of $A/Z$ so that $K/Z=(A/Z\times B/Z)\rtimes S$ and let $M/Z$ be the $K$-conjugate of $N/Z$ in $B/Z$. 

By Lemma \ref{lem:schur}, $\lambda$ extends to $\Irr(N)$ and by Gallagher's theorem there is a unique character $\alpha\in\Irr(N)$ of degree $3$ in $\Irr(N|\lambda)$. Since $H$ normalizes $N$ and $\lambda$ is $K$-invariant, it follows that $\alpha$ is $MN$-invariant and again by Lemma \ref{lem:schur} $\alpha$ extends to $MN\normal K$. By Gallagher's theorem, there are characters of three distinct degrees in $\Irr(MN|\alpha)\sbs\Irr(MN|\lambda)$, which contradicts $|\Irr(K|\lambda)|=2.$
\end{proof}

\begin{pro}\label{pro:imp4-3}
Assume that $S\cong \SSS_3$ or $\CCC_3$. Then  $|\Irr(K|\lambda)|>2$. In particular, Theorem \ref{thm:reduced} holds.
\end{pro}
\begin{proof} By \cite[p. 33]{V93} and \cite[Theorem 18.10]{Hup98} there is a $G$-orbit on $\Irr(D)$ of size $3(p^m-1)^2$ (here, we are considering the natural action of $G$ on $\Irr(D)$ induced by the action by conjugation on $D$). Now by Theorem \ref{lem:demeyerjanusz2} we have that the Sylow $q$-subgroups have square order, and the Sylow $2$-subgroups have order an odd power of $2$. Thus we may write $|G:ZD|=|K:Z|=2b^2$ for some positive integer $b$. If $\mu$ is a representative of the orbit of size $3(p^m-1)^2$, then
$$|G_\mu:ZD|=\frac{|G:ZD|}{|G:G_\mu|}=\frac{2b^2}{3(p^m-1)^2}$$ so $|G_\mu:ZD|_3$ can not be a square. This contradicts Lemma \ref{lem:contradiction}.
 \end{proof}

\subsection{The proof}\label{sec:proof of thm reduced} We can complete the proof of Theorem \ref{thm:A} by establishing Theorem \ref{thm:reduced}.

\textit{Proof of Theorem \ref{thm:reduced}.}
If $G/ZD$ is isomorphic to a subgroup of the semilinear group on $D$ then this is proved in Proposition \ref{pro:semilinear} regardless of the rank, so this handles case (i) of Theorems \ref{thm:passman}, \ref{thm:liebeck} and \ref{thm:foulser}. If $|D|=p^2$ then this is proved in Proposition \ref{pro:d=2}, which handles case (ii) of Theorems \ref{thm:passman}, \ref{thm:liebeck} and \ref{thm:foulser} and the cases (iv) of Theorems \ref{thm:liebeck} and \ref{thm:foulser} whenever $d=2$. If $G/ZD$ is as in case (iii) of Theorems \ref{thm:liebeck} and \ref{thm:foulser} then this is proved in Propositions \ref{pro:foulserex} and \ref{pro:liebeckex}. 
We are left with the case that $G/ZD$ is as in (iv) of Theorems \ref{thm:liebeck} and \ref{thm:foulser} when $d>2$. If the rank of $G/ZD$ on $D$ is 3 then see Proposition \ref{pro:imp3}. If the rank of $G/ZD$ on $D$ is 4 then see Propositions \ref{pro:imp4}, and \ref{pro:imp4-3}. This completes the proof. \qed

%%%

\end{document}